\documentclass{amsart} 

\usepackage{amsmath,amssymb,amsthm,amsfonts,amsxtra,mathtools}
\usepackage{enumerate,verbatim,mathrsfs,comment,color,url}
\usepackage[all,2cell,ps]{xy}
\usepackage[pagebackref]{hyperref}
\usepackage{cleveref}

\DeclareMathOperator{\Ass}{Ass}

\DeclareMathOperator{\indeg}{indeg}

\DeclareMathOperator{\Spec}{Spec}

\renewcommand{\ge}{\geqslant}
\renewcommand{\le}{\leqslant}

\newcommand{\bn}{\mathbb{N}}
\newcommand{\bz}{\mathbb{Z}}

\newcommand{\fa}{\mathfrak{a}}
\newcommand{\fm}{\mathfrak{m}}

\newcommand{\fp}{\mathfrak{p}}
\newcommand{\fq}{\mathfrak{q}}

\newcommand{\mca}{\mathcal{A}}
\newcommand{\mcb}{\mathcal{B}}

\newcommand{\xlra}{\xlongrightarrow}

\theoremstyle{plain}
\newtheorem{theorem}{Theorem}[section]
\newtheorem{lemma}[theorem]{Lemma}
\newtheorem{proposition}[theorem]{Proposition}
\newtheorem{corollary}[theorem]{Corollary}

\theoremstyle{definition}
\newtheorem{definition}[theorem]{Definition}

\newtheorem{example}[theorem]{Example}

\newtheorem{notation}[theorem]{Notation}

\newtheorem{question}[theorem]{Question}
\newtheorem{setup}[theorem]{Setup}

\theoremstyle{remark}
\newtheorem{remark}[theorem]{Remark}

\numberwithin{equation}{theorem}

\title[Asymptotic prime divisors and Vasconcelos invariant]{Asymptotic prime divisors and Vasconcelos invariant}
 
\author[D.~Ghosh]{Dipankar Ghosh}
\address{Department of Mathematics, Indian Institute of Technology Kharagpur, West Bengal - 721302, India}
\email{dipankar@maths.iitkgp.ac.in, dipug23@gmail.com}
\urladdr{\url{https://orcid.org/0000-0002-3773-4003}}

\author[R.~Nanduri]{Ramakrishna Nanduri}
\address{Department of Mathematics, Indian Institute of Technology Kharagpur, West Bengal - 721302, India}
\email{nanduri@maths.iitkgp.ac.in}

\author[S.~Pramanik]{Siddhartha Pramanik}
\address{Department of Mathematics, Indian Institute of Technology Kharagpur, West Bengal - 721302, India}
\email{siddharthap@kgpian.iitkgp.ac.in, pramaniksiddhartha2@gmail.com}
\urladdr{\url{https://orcid.org/0009-0005-1219-4265}}

\date{\today}

\begin{document}
    \pagenumbering{arabic}
    \thispagestyle{empty}

\begin{abstract}
Let $R$ be a Noetherian ring, $I$ an ideal of $R$, and $M$ a finitely generated $R$-module. In this article, we prove that
$$\displaystyle \mathrm{Ass}_R(M/I^n M) = \mathrm{Ass}_R(0:_M I) \cup  \mathrm{Ass}_R(I^{n-1} M/I^{n} M) \text{ for all }n \gg 0.$$
We then investigate the asymptotic behaviour of the (local) Vasconcelos invariant of $M/I^n M$ as a function of $n$, when $R$ is $\mathbb{N}$-graded, $I$ is homogeneous, and $M$ is $\mathbb{Z}$-graded. When $I$ is generated by elements of positive degree, we show that, for sufficiently large $n$, the (local) Vasconcelos invariant of $M/I^n M$ either coincides with that of the colon submodule $(0 :_M I)$, or is a polynomial in $n$ of degree one whose leading coefficient is one of the degrees of the generators of $I$. This dichotomy depends exclusively on two cases determined by $(0:_MI)$. Thus, we recover and considerably strengthen the main results of Fiorindo-Ghosh [Nagoya Math.~J. 258 (2025), 296–310.], where asymptotic linearity was shown under the additional assumption that $(0 :_M I)=0$.
\end{abstract}
\maketitle 

\section{Introduction}

In 1921, Emmy Noether \cite{N21} introduced the theory of primary decomposition, which profoundly shaped the development of Commutative Algebra and Algebraic Geometry. Many foundational results in these areas would not be possible without this theory, and it is in this sense that Noetherian rings acquire their central importance. From a modern perspective, however, the concept of associated primes has proven more fundamental than primary decomposition itself. Let $R$ be a (commutative) Noetherian ring, $I$ an ideal of $R$, and $M$ a finitely generated $R$-module. In the last few decades, the behaviour of various algebraic invariants of the module $M/I^n M$ for all large $n$ has been a widely studied subject. This was inspired by the influential result of Brodmann \cite{B79a}, which states that the set of associated primes $\Ass_R(M/I^nM) $ becomes independent of $n$ for all sufficiently large $n$. Along the way, he also proved that $\Ass_R(I^{n-1}M/I^{n}M)$ stabilizes for all $n \gg 0$. More generally, if $N$ is a submodule of $M$, then it follows from a result of Kingsbury-Sharp \cite[Thm.~1.5.(ii)]{KS96} that $\Ass_R(M/I^nN)$ become stable for all $n \gg 0$. We therefore fix the following notation for later use.

\begin{notation}\label{notation}
Let $N\subseteq M$ be finitely generated $R$-modules. Denote
   $$\mca^M_N(I) := \Ass_R \big(M/I^n N \big) \; \mbox{ and } \; \mcb_M(I) := \Ass_R \big(I^{n-1}M/I^{n} M \big) \text{ for all } n \gg 0.$$
 Set $ \mca_M(I) := \mca^M_M(I).$
\end{notation}

A natural question concerns the relationship between the stable sets $\mca_M(I)$ and $\mcb_M(I)$. In \cite[(7)]{B79a}, Brodmann showed that if $(0 :_M I) =0$, then $\mca_M(I) = \mcb_M(I)$, where $(0 :_M I) := \{ x \in M : Ix = 0\}$. In \cite[Cor.~13]{ME79}, McAdam-Eakin proved that the difference $\mca_R(I) \smallsetminus \mcb_R(I)$ is a subset of $\Ass_R(R)$. 
In this regard, we prove the following result, which generalizes the aforementioned results of Brodmann and McAdam-Eakin.

\begin{theorem}[See Theorem~\ref{th:ass} for stronger results]\label{th:intro1}
    Let $R$ be a Noetherian ring, $I$ an ideal of $R$, and $M$ a finitely generated $R$-module. Then, $$\Ass_R(M/I^n M) = \Ass_R(0:_M I) \cup  \Ass_R(I^{n-1} M/I^{n} M) \; \mbox{ for all } n \gg 0.$$
    In particular, $ \mca_M(I) = \Ass_R(0:_M I) \cup \mcb_M(I) $, cf.~\rm{Notation~\ref{notation}}.
\end{theorem}

Theorem~\ref{th:intro1} is a special case of the more general result proved in Theorem~\ref{th:ass}; see Remark~\ref{2.1-and-2.2}. In Theorem~\ref{th:ass}, in fact we establish that
$$\Ass_R(M/I^n N) \cap V(I) = \Ass_R(0:_M I) \cup  \Ass_R(I^{n-1} N/I^{n} N) \quad \mbox{for all } n \gg 0,$$
where $N\subseteq M$ is a submodule and $V(I) := \{\fp\in\Spec(R) : I \subseteq \fp\}$. Moreover, in Proposition~\ref{prop:ass}, we show that the set $\Ass_R(M/I^nN) \smallsetminus V(I)$ stabilizes from the very beginning, and that this stabilized set coincides with $\Ass_R (M/N) \smallsetminus V(I)$.

Note that $\mca_M(I) \smallsetminus \mcb_M(I) \subseteq \Ass_R(0:_M I)$. Also, the union in Theorem~\ref{th:intro1} may or may not be disjoint; see Examples~\ref{ex:1} and \ref{ex:2}, respectively. Thus, it is possible that $\mcb_M(I) \subsetneqq \mca_M(I)$.
However, in the graded setup, we show that certain local Vasconcelos invariants of $M/I^nM$ and $I^{n-1}M/I^{n}M$ eventually coincide. In the next part, we discuss this numerical invariant.

In recent years, the Vasconcelos invariant, also known as the $v$-number, which is related to the notion of associated primes, has garnered considerable attention from contemporary researchers. This invariant first appeared in \cite{CSTPV20}, where it was defined for homogeneous ideals in a polynomial ring over a field, in order to study the asymptotic behaviour of the minimum distance function of projective Reed-Muller type codes. Subsequently, Fiorindo-Ghosh provided the following module-theoretic definition. For the rest of this introduction, let $R$ be a Noetherian $\bn$-graded ring, $M$ a finitely generated graded $R$-module, and $I$ a homogeneous ideal of $R$. For $ n \in \bz$, the $n$th graded component of $M$ is denoted by $M_n$.

\begin{definition}{\cite[Def.~1.2]{FG25a}}
    For $\fp \in \Ass_R(M)$, the {\it local Vasconcelos invariant $($or the local $v$-number}$)$ of $M$ at $\fp$ is defined as
    $$v_{\fp}(M) := \inf \{ n : \fp = ( 0 :_R x) \text{ for some } x \in M_n \}.$$
    The {\it Vasconcelos invariant $($or $v$-number$)$} of $M$ is defined to be
    $$ v(M) := \inf \{v_{\fp}(M) : \fp \in \Ass_R(M) \}.$$
    By convention, $v(0) = \infty$.
\end{definition}

Motivated by the result of Brodmann \cite{B79a}, Conca proved in \cite[Thm.~1.1]{C24} that for each $ \fp \in \mca_R(I)$, the function $v_{\fp}(R/I^n) $ is {\it asymptotically linear} in $n$, i.e., $v_{\fp}(R/I^n) = a_{\fp}n + b_{\fp}$ for all $n\gg 0$, where the leading coefficient $a_\fp$ is equal to one of the degrees of the generators of $I$, and $b_{\fp} \in \bz$, provided that $R$ is a domain. The same result was independently proved by Ficarra-Sgroi in \cite[Thm.~3.1]{FS23} when $R$ is a polynomial ring over a field.
Fiorindo-Ghosh strengthen the result of Conca and Ficarra-Sgroi by proving the asymptotic linearity of $v_{\fp}(M/I^nM)$ for $\fp \in \mca_M(I)$ under the assumption that $ (0 :_M I) = 0$; see \cite[Thm.~2.14.(1)]{FG25a}. Note that $(0 :_R I) = 0$ when $R$ is a domain or a polynomial ring over a field. See also \cite{GP25,FG25b} for further generalizations of the result of Fiorindo-Ghosh. The asymptotic behaviour of this invariant has also been studied in \cite{BMS11, SK24, FS25, AS25a, KNS25, FM25, S24, AS25b, MP26} for various filtrations of ideals. It is shown in \cite[Thm.~2.11]{FG25a} that the functions $v(I^{n-1}M/I^nM)$ and $v_{\fp}(I^{n-1}M/I^nM)$ (for $\fp \in \mcb_M(I)$) are always eventually linear, with leading coefficients equal to one of the degrees of the generators of $I$. In addition, if $(0 :_M I) = 0$, then for $ \fp \in \mca_M(I)$, the functions $v_{\fp}(I^{n-1}M/I^nM)$ and $v_{\fp}(M/I^nM)$ coincide asymptotically; cf.~\cite[Thm.~2.14.(1)]{FG25a}.
Keeping these results in mind, the following questions arise naturally.

\begin{question}\label{ques-1}
    Let $(0 :_M I) \neq 0$.
    \begin{enumerate}[\rm (1)]
    \item 
    For each $\fp \in \mca_M(I)$, what is the asymptotic behaviour of $v_{\fp}(M/I^nM)$ as a function of $n$? When $\fp \in \mcb_M(I)$, do the functions $v_{\fp}(M/I^nM)$ and $v_{\fp}(I^{n-1}M/I^nM)$ eventually coincide?
    \item 
    What is the asymptotic behaviour of $v(M/I^nM)$? Is it true that $v(M/I^nM)$ agrees with $v(I^{n-1}M/I^nM)$ for all sufficiently large $n$?
    \end{enumerate}
\end{question}

In what follows, we provide complete answers to all the questions posed in \ref{ques-1}.


\begin{theorem}[See Theorem~\ref{th:main} for stronger results]\label{th:intro2}
Let $R$ be a Noetherian $\bn$-graded ring, $M$ a finitely generated $\bz$-graded $R$-module, and $I$ a homogeneous ideal of $R$. Suppose $I$ is generated by homogeneous elements $y_1,\dots, y_c$ of degree $d_1, \dots, d_c$, respectively. Then, the following hold:
\begin{enumerate}[\rm (1)]
    \item 
    Let $\fp \in \mca_M(I) \smallsetminus \Ass_R(0 :_M I)$. Then, $\fp \in \mcb_M(I)$. Moreover, there exist $a_{\fp} \in \{d_1, \dots, d_c \}$ and $b_{\fp} \in \bz$ such that 
    $$v_{\fp}(M/I^nM) = v_{\fp}(I^{n-1} M/I^n M) = a_{\fp}n + b_{\fp} \; \text{ for all } n \gg 0.$$
    \item 
    Let $\fp \in \mca_M(I) \cap \Ass_R(0 :_M I)$.
    If $d_i \ge 1$ for $ 1 \le i \le c$, then
    $$v_{\fp}(M/I^nM) = v_{\fp}(M) = v_{\fp}( 0:_M I) \; \text{ for all } n \gg 0.$$
    \item 
    Let $(0 :_M I) \neq 0$. Then, $v(M/I^n M) \le  v( 0 :_M I)$ for all $n \gg 0$. Furthermore, if $d_i \ge 1$ for $ 1 \le i \le c$, then
    \[
    v(M/I^n M) =  v( 0 :_M I) \; \mbox{ for all } n \gg 0.
    \]
    \item 
    Let $(0 :_M I) = 0$ and $\mca_M(I) \neq \emptyset$. Then
    $$v(M/I^nM) = v(I^{n-1}M/I^nM) \; \mbox{ for all } n \gg 0,$$
    which is asymptotically linear in $n$ with the leading coefficient in $ \{d_1, \dots, d_c \}$.
\end{enumerate}
\end{theorem}

Theorem~\ref{th:intro2} treats not only the case $(0 :_M I) \neq 0$ but also the situation in which $(0 :_M I) = 0$. Thus, the asymptotic behaviours of (local) Vasconcelos invariant of $M/I^nM$ are obtained in all possible cases. In part (2) (resp., (3)) of Theorem~\ref{th:intro2}, note that, eventually,
$v_{\fp}(M/I^nM) < v_{\fp}(I^{n-1}M/I^nM)$ (resp., $v(M/I^nM) < v(I^{n-1}M/I^nM))$,
whereas parts (1) and (4) describe the cases in which these functions eventually coincide. We construct various examples to complement our results. Example~\ref{ex:1} ensures that the equality in Theorem~\ref{th:intro2}.(3) may not occur if we do not assume that $I$ is generated by homogeneous elements of positive degree. We also show via an example that for $ \fp \in \mcb_M(I)$, the functions $v_{\fp}(M/I^nM)$ and $v_{\fp}(I^{n-1}M/I^nM)$ need not coincide eventually, cf.~Example~\ref{ex:2}. 

Theorem~\ref{th:intro2} follows from more general results proved in Theorem~\ref{th:main}, where the asymptotic behaviour of the (local) Vasconcelos invariant of $M/I^nN$ is studied for a graded submodule $N\subseteq M$. Whenever $ \fp \in \mca^M_N(I) \cap V(I)$,  we completely describe the asymptotic behaviour of $v_{\fp}(M/I^n N)$ by proving that $v_{\fp}(M/I^n N)$ is asymptotically either constant or a polynomial in $n$ of degree one. Moreover, these functions are asymptotically compared with those of $I^{n-1}N/I^nN$. In some cases, these two functions eventually coincide, showing that $v_{\fp}(M/I^nN)$ and $v(M/I^nN)$ eventually depend only on $I$ and $N $, and not on $M$. Surprisingly, in the remaining cases, $v_{\fp}(M/I^nN)$ and $v(M/I^nN)$ eventually do not depend on $N$. Thus, we recover and considerably strengthen the main results of \cite{FG25a}. See Theorem~\ref{th:main} and Remark~\ref{rmk:depend-on-M-N} for the details. When $\fp \in \mca^M_N(I)\smallsetminus V(I)$, we make some observations regarding the behaviour of $v_{\fp}(M/I^nN)$ in Proposition~\ref{prop:vnumber} and Remark~\ref{rmk:gen-ineq}.

We arrange the remaining sections of this article as follows. In Section~\ref{sec:ass}, we prove Theorem~\ref{th:ass}. In Section~\ref{sec:vnumber}, we establish several general lemmas and results leading to the proof of our main theorem, Theorem~\ref{th:main}. Finally, in Section~\ref{sec:examples}, we present concrete examples to illustrate our results.

\subsection*{Acknowledgments}
Pramanik thanks the Government of India for the financial support through the Prime Minister's Research Fellowship for his Ph.D.

\section{Asymptotic prime divisors}\label{sec:ass}

Throughout this section, $R$ is a Noetherian ring (not necessarily graded), $I$ is an ideal of $R$, and $N \subseteq M$ are finitely generated $R$-modules. In this section, we mainly prove Theorem~\ref{th:ass} and Proposition~\ref{prop:ass}. We begin with the following lemma.

\begin{lemma}\label{lem:ass}
    There exists $k \in \bn$ such that for any $n\ge k$, the following hold:
    \begin{enumerate}[\rm (1)]
    \item For $\fp \in V(I)$, we have
    $$\{ x \in M : \fp = (I^n N :_R x) \} \subseteq (0:_M I) + I^{n-1}N.$$
    \item 
    $ \Ass_R \big({M}/{I^n N} \big) \cap V(I) =  \Ass_R \big({ \left( (0 :_M I) + I^{n-1}N \right)}/{I^n N} \big).$
    \end{enumerate}    
\end{lemma}

\begin{proof}
    In view of \cite[Cor.~2.1]{D21}, there exists $k \in \bn$ such that 
    $$ (I^{n} N :_M I) = (0 :_M I) + I^{n-1} N \; \mbox{ for all }  n \ge k.$$
    
    (1) Let $n \ge k$ and $\fp \in V(I)$. If $ x \in M$ such that $\fp = ( I^n N :_R x)$, then 
    \begin{align*}
    x \in (I^n N :_M \fp) \subseteq (I^n N :_M I) =  (0 :_M I) + I^{n-1} N .
    \end{align*}
    Thus, $\{ x \in M : \fp = ( I^n N :_R x) \} \subseteq (0 :_M I) + I^{n-1} N$ for all $n \ge k$. 
    
    (2) Note that $(0:_M I) +I^{n-1} N$ is a submodule of $M$, and every associated prime of $\left( (0 :_M I) + I^{n-1} N \right)/ I^n N$ contains $I$. Therefore 
    $$\Ass_R\left({\big(( 0 :_M I) + I^{n-1}N\big)}/{I^n N} \right) \subseteq \Ass_R \left({M}/{I^n N} \right) \cap V(I) \mbox{ for all } n\ge 1.$$ 
    In order to show the other inclusion, let $n \ge k$ and $\fp \in \Ass_R \left({M}/{I^n N} \right) \cap V(I)$. Then, there exists $x \in M$ such that $(I^n N :_R x) = \fp$. Now, by (1), we obtain $ x \in (0 :_M I) + I^{n-1} N$. Consequently, the result follows.
\end{proof}

We now state and prove the main result of this section.

\begin{theorem}\label{th:ass}
    Assume {\rm Notation~\ref{notation}}. Then,
    $$\Ass_R(M/I^n N) \cap V(I) = \Ass_R(0:_M I) \cup  \Ass_R(I^{n-1} N/I^{n} N) \quad \mbox{for all } n \gg 0.$$
    In particular, $ \mathcal{A}_N^M(I) \cap V(I) = \Ass_R(0 :_M I) \cup \mcb_N(I)$.
\end{theorem}

\begin{proof}
    By the Artin-Rees lemma (cf.~\cite[17.1.6]{SH06}), we obtain $(0 :_M I) \cap I^n N =0$ for all $n \gg 0$. Thus, for all $n \gg 0$,
    \begin{equation}\label{2nd-iso-thm}
        \big(( 0:_M I) + I^n N \big)/I^n N \cong (0 :_M I)/(0 :_M I)\cap I^n N = (0 :_M I).   
    \end{equation}
    For each $n \ge 1$, we have the following short exact sequence:
    \begin{equation}\label{eq:ses2}
    0 \to \frac{I^{n-1} N}{I^n N} \to \frac{(0 :_M I) + I^{n-1} N}{I^n N} \to \frac{(0 :_M I) + I^{n-1} N}{I^{n-1} N} \to 0.
    \end{equation}
    In view of \eqref{2nd-iso-thm} and \eqref{eq:ses2}, for all $n \gg 0$, we obtain
    \begin{equation}\label{eq:ass}
    \Ass_R \big((( 0:_M I) + I^{n-1} N)/I^n N \big) \subseteq \Ass_R(0 :_M I) \cup \Ass_R \big(I^{n-1}N /I^n N \big).
    \end{equation}
    By \eqref{2nd-iso-thm}, since $(0 :_M I) \cong \big(( 0:_M I) + I^n N \big)/I^n N$ (for $n\gg 0$) and $ I^{n-1} N/I^n N$ (for $n\ge 1$) both are submodules of $ \big(( 0:_M I) + I^{n-1} N \big)/I^n N$, the reverse inclusion in \eqref{eq:ass} is immediate. Consequently, the proof follows from Lemma~\ref{lem:ass}.(2).
\end{proof}   

\begin{remark}\label{2.1-and-2.2}
    For each $n \ge 1$, note that $\Ass_R(M/I^n M) \subseteq V(I)$. As a result, Theorem~\ref{th:intro1} follows from Theorem~\ref{th:ass} by taking $N=M$.
\end{remark}

Note that Theorem~\ref{th:ass} implies that $\Ass_R(0:_M I)\subseteq \Ass_R(M/I^n N)$ for all $n\gg 0$. This also follows from a more general fact mentioned in the following.


\begin{lemma}\label{lem:inj}
    The canonical homomorphism from $(0 :_M I)$ to $M/ I^n N$ defined by $ x \mapsto x + I^n N$, is injective, for all $n \gg 0$.
\end{lemma}

\begin{proof}
    Let $M' := (0 :_M I)$. Then, by the Artin-Rees lemma, $I^n N \cap M'=0$ for all $n \gg 0$. Therefore, since the kernel of the canonical homomorphism from $M'$ to $M/ I^n N$ is given by $I^n N \cap M'$, the lemma follows.
\end{proof}

Keeping in mind the hypotheses of Theorem~\ref{th:intro2}, we now describe, in the following lemma, which prime ideals in $\Ass_R(M)$ belong to $\Ass_R(0 :_M I)$. This simple lemma is used repeatedly in the next section. Set $\Gamma_I(M) : = \bigcup_{n \ge 1}(0 :_M I^n)$.

\begin{lemma}\label{lem:ass-Gamma_I-M}
    $\Ass_R(M) \cap V(I) = \Ass_R(\Gamma_I(M)) = \Ass_R(0 :_M I^m)$ for every $m\ge 1$.
\end{lemma}

\begin{proof}
    The first equality is trivial. We show the second equality. Fix $m\ge 1$. Note that $\Ass_R(0 :_M I^m) \subseteq \Ass_R(\Gamma_I(M))$ as $(0 :_M I^m) \subseteq \Gamma_I(M)$. Let $\fp \in \Ass_R(\Gamma_I(M))$. Then, by the first equality, we obtain $ \fp = (0 :_R x)$ for some $ x \in M$ and $I \subseteq \fp$. Thus $x \in (0 :_M I) \subseteq (0 :_M I^m)$. Consequently, $ \fp \in \Ass_R(0 :_M I^m)$.
\end{proof}

As a consequence of Theorem~\ref{th:ass}, we obtain the following result.

\begin{corollary}\label{cor:ass}
    Assume \rm{Notation~\ref{notation}}. Then
    $$\Ass_R(I^{n-1}N/I^n N) = \Ass_R \big(M/(\Gamma_I(M) + I^n N) \big) \cap V(I) \quad \text{for all } n \gg 0.$$
    In particular, $ \mcb_N(I) = \mca^{\overline{M}}_{\overline{N}}(I) \cap V(I)$, where $\overline{(-)}$ denotes modulo $\Gamma_I(M)$.
\end{corollary}

\begin{proof}
    Note that $\overline{M} = M/\Gamma_I(M)$ and $I^n\overline{N} = (I^nN + \Gamma_I(M))/\Gamma_I(M)$, hence $\overline{M}/I^n \overline{N} \cong M/(\Gamma_I(M) + I^n N)$ for all $n\ge 0$. Since $(0:_{\overline{M}}I) = 0$, by Theorem~\ref{th:ass}, we obtain $ \mca^{\overline{M}}_{\overline{N}}(I) \cap V(I) = \mcb_{\overline{N}}(I) $. Now, by the Artin-Rees lemma, it follows that $\Gamma_I(M) \cap I^n N = 0$ for all $n \gg 0$. We claim that the map $x + I^n N \mapsto x + (I^n N + \Gamma_I(M))$ from $I^{n-1}N / I^n N$ to $\big(I^{n-1} N + \Gamma_I(M) \big)/ \big(I^{n} N + \Gamma_I(M) \big)$ is an isomorphism for all $n \gg 0$. It suffices to show that it is injective for all $n \gg 0$. Let $x \in I^{n-1} N$ be such that $x \in I^n N + \Gamma_I(M)$, then $x = y + z$ for some $y \in I^n N$ and $ z \in \Gamma_I(M)$. So, it follows that $ x-y= z \in \Gamma_I(M) \cap I^{n-1}N =0$, provided $n$ is sufficiently large. Thus $ x = y \in I^n N$. Therefore, for all $n \gg 0$, we have 
    \begin{align*}
    \frac{I^{n-1} \overline{N}}{I^{n} \overline{N}} 
    \cong \frac{\big(I^{n-1} N + \Gamma_I(M) \big)}{\big(I^{n} N + \Gamma_I(M) \big)}  
    \cong \frac{I^{n-1} N}{I^{n}N},
    \end{align*}
    hence $\Ass_R(I^{n-1}\overline{N}/I^n \overline{N}) = \Ass_R(I^{n-1}N/I^n N)$. So $\mcb_{\overline{N}}(I) = \mcb_N(I)$.
\end{proof}

\begin{remark}\label{rmk:isomorphism}
    With the notation as in Corollary~\ref{cor:ass}, its proof shows that
    $$\frac{I^{n-1} \overline{N}}{I^{n} \overline{N}} \cong \frac{I^{n-1} N}{I^{n}N} \; \mbox{ for all } n \gg 0.$$
    We will use this isomorphism again in the proof of Theorem~\ref{th:main1}.
\end{remark}

The following result is of interest, as it shows that the primes in $\Ass_R(M/I^nN)$ that do not contain $I$ stabilize from the very beginning.

\begin{proposition}\label{prop:ass}
    $\Ass_R(M/I^nN) \smallsetminus V(I) = \Ass_R(M/N) \smallsetminus V(I)$ for all $n \ge 0$.
\end{proposition}

\begin{proof}
    The equality holds trivially when $n=0$. Let $\fp \in \Ass_R (M/N) \smallsetminus V(I)$. Then $\fp = (N :_R x)$ for some $x \in M$. Fix $n \ge 1$ and choose an element $r \in I\smallsetminus \fp$. Clearly, $\fp \subseteq (I^n N :_R r^n x)$. Let $ a \in (I^n N :_R r^n x)$. Then $r^n xa \in I^n N$, and hence
    $$ r^n a \in (I^n N :_R x) \subseteq (N :_R x) = \fp,$$
    which implies that $ a \in \fp$. Consequently, $\fp = (I^n N :_R r^n x)$ for every $n \ge 1$. Thus,
    $$ \Ass_R (M/N) \smallsetminus V(I) \subseteq \Ass_R(M/I^nN) \smallsetminus V(I) \text{ for every } n \ge 1.$$
    To show the reverse inclusion, let $n \ge 1$ and $\fp \in \Ass_R(M/I^n N) \smallsetminus V(I)$. Consider the short exact sequence
    \begin{equation}\label{eq:sesassi}
    0 \to N/I^n N \to M/I^n N \to M/N \to 0.
    \end{equation}
    Since every prime ideal in $\Ass_R (N/ I^n N)$ must contain $I$, we have $\fp \notin \Ass_R (N/ I^n N )$. Therefore, from the short exact sequence \eqref{eq:sesassi}, it follows that $\fp \in \Ass_R (M/N)$. This completes the proof.
\end{proof}

\begin{remark}
    The conclusions of Theorem~\ref{th:ass} and Proposition~\ref{prop:ass} continue to hold if the assumption that $R$ is Noetherian is replaced by the condition that $M$ is Noetherian, since this reduction can be carried out using the same idea as in the proof of \cite[Thm.~1.6(i)]{KS96}.
\end{remark}



\begin{remark}
    Using the theory of primary decomposition, one may observe that $\mca^M_N(I) \subseteq V(I)$ if and only if $I^k M \subseteq N$ for some $k \ge 1$.
\end{remark}




\section{Asymptotic Vasconcelos invariant}\label{sec:vnumber}

In this section, we present results concerning the Vasconcelos invariant and therefore work in the graded setup. The main goal of this section is to prove Theorem~\ref{th:main}, in which the asymptotic behaviour of $v_{\fp}(M/I^n N)$ is completely described under the assumption that $\fp \in \mca^M_N(I) \cap V(I)$. At the end of this section, we make some comments on the behaviour of $v_{\fp}(M/I^nN)$ when $\fp \in \mca^M_N(I)\smallsetminus V(I)$; see Proposition~\ref{prop:vnumber} and Remark~\ref{rmk:gen-ineq}. We begin by establishing some general results.

For a short exact sequence of finitely generated graded modules, we have the following result, which compares the local $v$-numbers of the modules.

\begin{proposition}\label{prop:sesvnumber}
Let $0 \to M' \xlra{\phi} M \xlra{\psi} M'' \to 0$ be a short exact sequence of finitely generated $\bz$-graded modules over a Noetherian $\bn$-graded ring $R$.
\begin{enumerate}[\rm (1)]
\item 
If $\fp \in \Ass_R(M')$, then $v_{\fp}(M) \le v_{\fp}(M')$. Consequently, $v(M) \le v(M')$. 
\item 
If $\fp \in \Ass_R(M) \smallsetminus \Ass_R(M')$, then $v_{\fp}(M'') \le v_{\fp}(M)$.
\end{enumerate}
\end{proposition}

\begin{proof}
    (1) This is proved in \cite[Prop.~2.5]{FG25a} for the case where $M'$ is a graded submodule of $M$. Hence, the assertions in (1) follow. (2) Let $\fp \in \Ass_R(M) \smallsetminus \Ass_R(M')$. Then $\fp = (0 :_R x)$ for some $x \in M$ such that $v_{\fp}(M) = \deg(x)$. Note that for $r \in R$ if $rx \neq 0$, then $\fp = ( 0 :_R rx)$. Therefore, since $\fp \notin \Ass_R(M')$, we must have $\phi(M') \cap Rx =0$. Hence, it follows that $\fp = (0 :_R \psi(x))$. Thus, $v_{\fp}(M'') \le \deg(\psi(x)) = \deg(x) = v_{\fp}(M)$.
\end{proof}

By Lemma~\ref{lem:ass-Gamma_I-M}, we have $ \Ass_R(0 :_M I) = \Ass_R(\Gamma_I(M)).$ We now show that the modules $(0 :_M I)$ and $\Gamma_I(M)$ in fact have the same (local) Vasconcelos invariant.

\begin{proposition}\label{prop:vnumberchain}
    Let $R$ be a Noetherian $\bn$-graded ring, $M$ a finitely generated $\bz$-graded $R$-module, and $I$ a homogeneous ideal of $R$.
    \begin{enumerate}[\rm (1)]
    \item 
    If $\fp \in \Ass_R(0:_M I)$, then $v_{\fp}(M) = v_{\fp}(\Gamma_I(M)) = v_{\fp}(0 :_M I^n)$ for each $n\ge 1$.
    \item
    $v(\Gamma_I(M)) = v(0 :_M I^n)$ for each $ n \ge 1$.
    \item 
    If $\Ass_R(M) \subseteq V(I)$, then $v(M)=v(\Gamma_I(M))$.
    \end{enumerate}
\end{proposition}

\begin{proof}
    Let $\fp \in \Ass_R(0:_M I)$. Consider the chain of submodules
    $$ (0 :_M I) \subseteq (0 :_M I^2 ) \subseteq (0 :_M I^3) \subseteq \cdots \subseteq \Gamma_I(M) \subseteq M.$$
    Hence, by Proposition~\ref{prop:sesvnumber}.(1), we obtain
    \begin{equation}\label{eq:vnumberchain}
    v_{\fp}(M) \le v_{\fp}(\Gamma_I(M)) \le v_{\fp}(0 :_M I^{n+1}) \le v_{\fp}(0 :_M I^n) \quad \text {for each } n\ge 1.
    \end{equation}
    Note that $I\subseteq \fp$ since $\fp \in \Ass_R(0:_M I)$. Therefore, if $\fp = (0 :_R x)$ for some $x \in M$, then clearly $x \in (0 :_M I)$. Thus, it follows that $v_{\fp}(0 :_M I) \le v_{\fp}(M)$. Combining this inequality with the inequalities in \eqref{eq:vnumberchain}, we get that
    $$ v_{\fp}(M) = v_{\fp}(\Gamma_I(M)) = v_{\fp}(0 :_M I^n)\quad \text {for each } n\ge 1.$$
    This proves (1). Hence, in view of Lemma~\ref{lem:ass-Gamma_I-M}, the last two parts follow.
\end{proof}

In the rest of this section, we shall work with the following setup.

\begin{setup}\label{setup}
    Let $R$ be a Noetherian $\bn$-graded ring. Let $N \subseteq M$ be finitely generated $\bz$-graded $R$-modules. Suppose $I$ is an ideal of $R$ generated by homogeneous elements $y_1,\dots, y_c$ of degrees $d_1, \dots , d_c$ respectively.
\end{setup}

By Lemma~\ref{lem:ass}, for all $n \gg 0$, the associated primes of $M/I^nN$ that contain $I$ coincide with the associated primes of $\left( (0 :_M I) +I^{n-1}N \right)/I^n N$. Next, we show that these two modules have the same asymptotic (local) Vasconcelos invariant.

\begin{proposition}\label{prop:vnumbercompare}
    With \textnormal{Setup~\ref{setup}} and \textnormal{Notation~\ref{notation}}, let $\fp \in \mca^M_N(I)\cap V(I)$. Then, 
    $$v_{\fp}(M/I^n N) = v_{\fp}\left((( 0 :_M I) + I^{n-1}N)/I^n N \right) \; \text{ for all } n \gg 0.$$
\end{proposition}

\begin{proof}
In view of Lemma~\ref{lem:ass}.(2) and Proposition~\ref{prop:sesvnumber}.(1), it follows that 
$$ v_{\fp}(M/I^n N) \le v_{\fp}\left((( 0 :_M I) + I^{n-1}N)/I^n N \right) \; \text{ for all } n \gg 0.$$
The reverse inequality holds by Lemma~\ref{lem:ass}.(1).
\end{proof}

When $(0 :_M I) \neq 0$, the following result guarantees a constant upper bound for the (local) Vasconcelos invariant(s) of $M/I^n N$.

\begin{proposition}\label{prop:vbound}
    With \textnormal{Setup~\ref{setup}}, if $\fp \in \Ass_R(0:_M I)$, then 
    $$ v_{\fp}\left( M/{I^n N}\right) \le v_{\fp}(0 :_M I) = v_{\fp}(M) \text{ for all } n \gg 0. $$
    Consequently, $v(M/I^n N) \le v(0:_M I)$ for all $n \gg 0$.
\end{proposition}

\begin{proof}
    By Lemma~\ref{lem:inj}, the module $(0:_M I)$ is isomorphic to a submodule of $M/I^nN$ for all $n \gg 0$. Then the first inequality follows from Propositions~\ref{prop:sesvnumber}.(1), while the equality $v_{\fp}(0 :_M I) = v_{\fp}(M)$ is shown in Proposition~\ref{prop:vnumberchain}.(1).
\end{proof}

Next, we show that $v_{\fp}(M/I^n N)$ is asymptotically linear in $n$ for
$$\fp \in \big(\mca^M_N(I) \cap V(I)\big) \smallsetminus \Ass_R( 0 :_M I).$$
This considerably strengthens the main result of \cite{FG25a}, where asymptotic linearity is proved under the additional assumption that $(0 :_M I)=0$; see \cite[Thm.~1.9.(1)]{FG25a}.

\begin{theorem}\label{th:main1}
    With \textnormal{Setup~\ref{setup}} and \textnormal{Notation~\ref{notation}}, let $\fp \in \mathcal{A}_N^M(I)\smallsetminus \Ass_R( 0 :_M I)$ such that $I \subseteq \fp$. Then $ \fp \in \mcb_N(I)$. Moreover, there exist $a_{\fp} \in \{d_1, \dots, d_c \}$ and $b_{\fp} \in \bz$ such that 
    $$ v_{\fp}(M/I^n N) = v_{\fp}(I^{n-1} N/ I^n N) = a_{\fp}n + b_{\fp} \; \text{ for all } n \gg 0.$$
\end{theorem}

\begin{proof}
    Since $\fp \in \mca^M_N(I) \smallsetminus \Ass_R(0 :_M I)$ and $I \subseteq \fp$, by Theorem~\ref{th:ass} and Corollary~\ref{cor:ass}, it follows that 
    $$\fp \in \mcb_N(I) \subseteq \Ass_R \big(M/(\Gamma_I(M) + I^n N) \big) \; \mbox{ for all } n\gg 0.$$
    Hence, in view of \cite[Thm.~2.11]{FG25a}, there exist $a_{\fp} \in \{d_1, \dots,d_c\}$ and $b_{\fp} \in \bz$ such that
    \begin{equation}\label{FG-2.11-linear}
        v_{\fp}(I^{n-1}N/I^n N) = a_{\fp}n+b_{\fp} \; \mbox{ for all } n\gg 0.
    \end{equation}
    Therefore, by Proposition~\ref{prop:vnumbercompare}, if $( 0:_M I) =0 $, then
    $$v_{\fp}(M/I^n N) = v_{\fp}(I^{n-1} N/I^n N)= a_{\fp}n+b_{\fp} \; \mbox{ for all } n \gg 0.$$
    So, we may assume that $(0 :_M I) \neq 0$. Let $\overline{(-)}$ denote reduction modulo $\Gamma_I(M)$. Then $(0 :_{\overline{M}} I) = 0$. Consequently, Proposition~\ref{prop:vnumbercompare} yields that
    \begin{equation}\label{eq:th3.6i}
    v_{\fp} \big(M/(\Gamma_I(M) + I^n N) \big) 
    = v_{\fp} \big(\overline{M}/I^n \overline{N} \big) 
    = v_{\fp} \big(I^{n-1}\overline{N}/I^n \overline{N} \big) \; \mbox{ for all } n \gg 0.
    \end{equation}
    In view of Remark~\ref{rmk:isomorphism}, $\frac{I^{n-1} \overline{N}}{I^{n} \overline{N}} \cong \frac{I^{n-1} N}{I^{n}N}$ for all $n \gg 0$. Thus, since $\fp \in \mcb_N(I)$,
    \begin{equation}\label{eq:th3.6ii}
    v_{\fp} \big(I^{n-1}\overline{N}/I^n \overline{N} \big) = v_{\fp} \big(I^{n-1}N/I^n N \big) \; \text{ for all } n\gg 0.
    \end{equation} 
    By the Artin-Rees lemma, $\Gamma_I(M) \cap I^n N = 0$ for all $n\gg 0$. So
    \begin{center}
        $(\Gamma_I(M) + I^n N)/ I^n N \cong \Gamma_I(M)/(\Gamma_I(M)\cap I^n N) = \Gamma_I(M)$ \;for all $n\gg 0$.
    \end{center}
    Therefore, since $\fp \notin \Ass_R( 0 :_M I) = \Ass(\Gamma_I(M))$ (cf.~Lemma~\ref{lem:ass-Gamma_I-M}), it follows that
    $$\fp \notin \Ass_R \big((\Gamma_I(M) + I^n N)/ I^n N \big) \; \mbox{ for all } n \gg 0.$$
    Now, for every $n \ge 1$, consider the short exact sequence 
    \begin{equation*}
    0 \to \frac{\Gamma_I(M) +I^n N}{I^n N} \to \frac{M}{I^n N} \to \frac{M}{\Gamma_I(M) + I^n N} \to 0.
    \end{equation*}
    In view of Proposition~\ref{prop:sesvnumber}.(2), we obtain
    \begin{equation}\label{eq:th3.6iii}
    v_{\fp} \big(M/(\Gamma_I(M) + I^n N) \big) \le v_{\fp}(M/I^n N) \le v_{\fp}(I^{n-1} N/ I^n N) \; \mbox{ for all } n \gg 0,
    \end{equation}
    where the second inequality is obtained using Proposition~\ref{prop:sesvnumber}.(1). Hence, our assertion follows from \eqref{FG-2.11-linear}, \eqref{eq:th3.6i}, \eqref{eq:th3.6ii}, and \eqref{eq:th3.6iii}.
\end{proof}

We further investigate the asymptotic behavior of $v_{\fp}(M/I^n N)$ as a function of $n$, where $\fp \in \Ass_R(0 :_M I)$. Since the sets $\Ass_R(0 :_M I)$ and $\mcb_N(I)$ need not be disjoint (cf.~Example~\ref{ex:2}), there are two cases:  $\fp \in \mcb_N(I)$ or $\fp \notin \mcb_N(I)$. In the following theorem, we treat the case $\fp \notin \mcb_N(I)$.

\begin{theorem}\label{th:main2i}
    With \textnormal{Setup~\ref{setup}} and \textnormal{Notation~\ref{notation}}, let $\fp \in \mathcal{A}_N^M(I)\smallsetminus \mathcal{B}_N(I)$. Suppose $I \subseteq \fp$. Then, 
    $$v_{\fp}(M/I^n N) = v_{\fp}( 0 :_M I) = v_{\fp}(M) \; \text{ for all } n \gg 0.$$
\end{theorem}

\begin{proof}
    Since $\fp \in \mathcal{A}_N^M(I)\smallsetminus \mathcal{B}_N(I)$ and $I \subseteq \fp$, Theorem~\ref{th:ass} yields $\fp \in \Ass_R(0 :_M I)$, and Lemma~\ref{lem:ass}.(2) shows that
    $$\fp \in \Ass_R \left( \frac{(0 :_M I) + I^{n-1}N}{I^{n} N} \right) \; \mbox{ for all } n \gg 0. $$
    As $\fp \notin \mcb_N(I)$, in view of the short exact sequence \eqref{eq:ses2}, by Proposition~\ref{prop:sesvnumber}.(2), we obtain 
    \begin{equation} \label{eq:vnumberineq}
    v_{\fp} \left( \frac{(0 :_M I) + I^{n-1} N}{I^{n-1} N} \right) 
    \le v_{\fp} \left( \frac{(0 :_M I) + I^{n-1}N}{I^{n} N} \right) \;
    \mbox{ for all } n\gg 0.
    \end{equation}
    Now, by the Artin-Rees lemma, $(0 :_M I) \cap I^n N = 0$ for all $n \gg 0$. Therefore, 
    $${\big((0 :_M I) + I^{n}N \big)}/{I^{n} N} \cong (0 :_M I)/\big( (0 :_M I) \cap I^{n}N \big) = (0 :_M I) \; \mbox{ for all } n \gg 0.$$ 
    Thus, for all $n \gg 0$, we get
    \begin{align*}
    v_{\fp}(0 :_M I) 
    & = v_{\fp} \left( \frac{(0 :_M I) + I^{n-1}N}{I^{n-1} N} \right)\\
    & \le v_{\fp} \left( \frac{(0 :_M I) + I^{n-1}N}{I^{n} N} \right) \quad \text{[by \eqref{eq:vnumberineq}]}\\
    & = v_{\fp} \left( M/I^n N \right) \quad \text{[by Proposition~\ref{prop:vnumbercompare}]}\\
    & \le v_{\fp}(0 :_M I) \quad \text{[by Proposition~\ref{prop:vbound}]}.   
    \end{align*} 
    This implies the first equality. The second equality $v_{\fp}( 0 :_M I) = v_{\fp}(M)$ follows from Proposition~\ref{prop:vnumberchain}.(1).
\end{proof}

Next, we consider the case $\fp \in \mcb_N(I)$. For a graded module $M$, we denote by
$$\indeg(M) := \inf \{ n \in \bz : M_n \neq 0 \}, $$ 
the \emph{initial degree} of $M$. By convention, $\indeg(0) = \infty$.

\begin{theorem}\label{th:main2ii}
    With \textnormal{Setup~\ref{setup}} and \textnormal{Notation~\ref{notation}}, let $\fp \in \mcb_N(I)\cap \Ass_R(0 :_M I)$. Suppose $d_i \ge 1$ for $1 \le i \le c$. Then 
    $$v_{\fp}(M/I^n N) = v_{\fp}( 0 :_M I) = v_{\fp}(M) \; \text{ for all } n \gg 0.$$
\end{theorem}

\begin{proof}
    In view of Proposition~\ref{prop:vbound}, it suffices to show that
    \begin{equation}\label{eq:vnumberlbound}
    v_{\fp}(0:_M I) \le v_{\fp}(M/I^n N) \; \text{ for all } n \gg 0. 
    \end{equation}
    Let $ k_1 \ge 1$ be such that $ \fp \in \Ass_R(I^{n-1} N/I^n N)$ for all $n \ge k_1$. Then, for each $ n\ge k_1$, we have $I^{n-1} N/I^n N \neq 0$, and hence $ I^n N \neq 0$. Set $I^n N := 0$ for $n < 0$. Note that the module $\mathscr{L} : = \bigoplus_{n \in \bz} I^n N$ is a finitely generated $\bz^2$-graded module over the $\bn^2$-graded ring $T := \bigoplus_{n \in \bn} I^n$, where the $(m,n)$th graded components of $\mathscr{L}$ and $T$ are $(I^n N)_m$ and $(I^n)_m$, respectively, for every $(m,n) \in \bz^2$. So, in view of \cite[Thm.~2.8.(1)]{FG25a}, there exist $k_2 \ge 1$, $a \in \{d_1, \dots, d_c \}$ and $b \in \bz$ such that 
    $$\indeg(I^n N) = an + b \; \mbox{ for all }  n \ge k_2.$$ 
    Since $d_i \ge 1$ for $1 \le i \le c$, the function $\indeg(I^n N)$, for $n \ge k_2$, is strictly increasing. Choose $ k_3 \ge k_2$ such that $ \indeg(I^{k_3} N) > v_{\fp}(M)$. Now, by the Artin-Rees lemma, there exists $k_4 \ge 1$ such that $(0 :_M I) \cap I^n N = 0$ for all $n \ge k_4$. By Lemma~\ref{lem:ass}.(1), we choose $k > \max \{ k_1, k_3, k_4 \}$ such that
    \begin{equation}\label{2.1-subset}
        \{ x \in M : \fp = ( I^n N :_R x) \} \subseteq (0 :_M I) + I^{n-1} N \; \text{ for all } n \ge k.
    \end{equation}
    We claim that \eqref{eq:vnumberlbound} holds for all $n \ge k$. In order to prove the claim, fix $n \ge k$. Let $x \in M$ be a homogeneous element such that $\fp = (I^n N :_R x)$ and $ \deg(x) = v_{\fp}(M/I^n N)$. Then, by \eqref{2.1-subset}, $x \in (0 :_M I) + I^{n-1} N$. So $ x = y + z$ for some $y \in (0 :_M I)$ and $ z \in I^{n-1} N$. Since $(0 :_M I) \cap I^{n-1} N = 0$, the module $(0 :_M I) + I^{n-1} N$ is the (internal) direct sum of $(0 :_M I)$ and $I^{n-1}N$. As $x$ is homogeneous, $y$ and $z$ must also be homogeneous elements of degree equal to $\deg(x)$. Hence, in view of Proposition~\ref{prop:vbound}, we obtain 
    $$\deg(z) = \deg(x) =v_{\fp}(M/I^n N) \le v_{\fp}(M) < \indeg(I^{n-1} N),$$ 
    which implies that $z=0$. Consequently, $x = y \in (0 :_M I)$. Note that
    $$\fp = (I^n N :_R x) \supseteq (0 :_R x).$$
    Let $r \in \fp$. Then $rx \in (0:_M I) \cap I^nN = 0$ since $x \in (0 :_M I)$. Thus $r\in (0 :_R x)$. This shows that $\fp = (0 :_R x)$. Therefore, $v_{\fp}(0 :_M I) \le \deg(x) =v_{\fp}(M/I^n N)$. This completes the proof.
\end{proof}

\begin{remark}\label{rmk:eventually-cons}
    In Theorem~\ref{th:ass}, it is shown that
    \[
    \mathcal{A}_N^M(I) \cap V(I) = \Ass_R(0 :_M I) \cup \mcb_N(I).
    \]
    Let $\fp \in \Ass_R( 0 :_M I)$. Then there exists a homogeneous element $ x \in (0 :_M I) $ such that $\fp = (0 :_R x) $ and $\deg(x) = v_{\fp}( 0 :_M I)$. In view of Lemma~\ref{lem:inj}, we also have $\fp = (I^n N :_R x)$. Hence, by Theorems~\ref{th:main2i} and \ref{th:main2ii},
    $$v_{\fp}(M/I^n N) = v_{\fp}( 0 :_M I) = v_{\fp}(M) = \deg(x) \; \text{ for all } n \gg 0.$$
\end{remark}

Next, we give a description of the asymptotic behaviour of the Vasconcelos invariant of $M/I^nN$, which depends on whether the module $(0 :_M I)$ is zero or not.

\begin{theorem}\label{th:main3and4}
    Assume \textnormal{Setup~\ref{setup}} and \textnormal{Notation~\ref{notation}}. Then the following hold:
    \begin{enumerate}[\rm(1)]
    \item 
    Let $(0 :_M I) \neq 0$. Then, $v \big(M/I^n N \big) \le v( 0 :_M I) \text{ for all } n\gg 0.$
    Furthermore, if $d_i \ge 1$ for $1 \le i \le c$, and if $\mca^M_N(I) \subseteq V(I)$, then
    $$
    v \big(M/I^n N \big) =  v( 0 :_M I) \; \mbox{ for all } n \gg 0.
    $$
    \item 
    Let $( 0 :_M I) = 0$ and $\emptyset \neq \mca^M_N(I) \subseteq V(I)$. Then, 
    $$v\big(M/I^n N\big) = v\big(I^{n-1} N/I^n N\big) \; \mbox{ for all } n \gg 0,$$ 
    which is linear in $n$ with the leading coefficient in $\{ d_1, \dots, d_c \}$.
    \end{enumerate}
\end{theorem}

\begin{proof}
    (1) The first part is shown in Proposition~\ref{prop:vbound}. For the second part, assume that $d_i \ge 1$ for $ 1 \le i \le c$, and $\mca^M_N(I) \subseteq V(I)$. Then, by Theorem~\ref{th:main1}, the function $v_{\fp}(M/I^n N)$ is strictly increasing for all $n\gg 0$, where $\fp \in \mca^M_N(I) \smallsetminus \Ass_R(0 :_M I)$. On the other hand, when $\fp \in \Ass_R(0 :_M I)$, we have $v_{\fp}(M/I^n N) = v_{\fp}(0 :_M I)$ for all $n\gg 0$, cf.~Remark~\ref{rmk:eventually-cons}. Note that $\emptyset \neq \Ass_R(0 :_M I) \subseteq \mca^M_N(I)$ (by Theorem~\ref{th:ass}). Thus, for all $n \gg 0$, it follows that
    \begin{align*}
    v(M/I^n N) 
    & = \min \{ v_{\fp}(M/I^n N) : \fp \in \Ass_R(0 :_M I) \} \\
    & = \min \{ v_{\fp}(0 :_M I) : \fp \in \Ass_R(0 :_M I) \} 
    = v(0 :_M I).
    \end{align*}

    (2) Since $(0 :_M I) = 0$ and $\mca^M_N(I) \subseteq V(I)$, by Theorem~\ref{th:ass}, we obtain $\mca^M_N(I) = \mcb_N(I)$. Hence Theorem~\ref{th:main1} yields that
    \begin{align*}
    v(M/I^n N) 
    & = \min\{v_{\fp}(M/I^n N) : \fp \in \mca^M_N(I) \} \\
    & = \min \{ v_{\fp}(I^{n-1} N/I^n N) : \fp \in \mcb_N(I)\} \\
    & = v(I^{n-1}N/I^n N) \; \mbox{ for all } n \gg 0.
    \end{align*} 
    Now, it is known that for every $\fp \in \mcb_N(I)$, the function $v_{\fp}(I^{n-1} N/I^n N)$ is asymptotically linear in $n$ with the leading coefficient in $\{ d_1, \dots, d_c \}$, see \cite[Thm.~2.11]{FG25a}. As the minimum of finitely many linear functions in $n$ is also asymptotically linear in $n$ with leading coefficient equal to the minimum of those leading coefficients, the desired result follows.
\end{proof}

Now we are in a position to present the main results of this section.

\begin{theorem}\label{th:main}
    Assume \textnormal{Setup~\ref{setup}} and \textnormal{Notation~\ref{notation}}. Then the following hold:
    \begin{enumerate}[\rm(1)]
    \item 
    Let $\fp \in \mca^M_N(I) \smallsetminus \Ass_R(0 :_M I)$ and $I \subseteq \fp$. Then, $ \fp \in \mcb_N(I)$. Moreover, there exist $a_{\fp} \in \{d_1, \dots, d_c \}$ and $b_{\fp} \in \bz$ such that 
    $$v_{\fp} \big(M/I^n N\big) = v_{\fp} \big(I^{n-1} N/I^n N \big) = a_{\fp}n + b_{\fp} \; \text{ for all } n \gg 0.$$
    \item 
    Let $\fp \in \mca^M_N(I) \cap \Ass_R(0 :_M I)$. If $d_i \ge 1$ for $ 1 \le i \le c$, then 
    $$v_{\fp} \big(M/I^n N \big) = v_{\fp}(M) = v_{\fp}( 0:_M I) \; \text{ for all } n \gg 0.$$
    \item 
    Let $(0 :_M I) \neq 0$. Then, $v \big(M/I^n N \big) \le  v( 0 :_M I)$ for all $n \gg 0$. Furthermore, if $d_i \ge 1$ for $ 1 \le i \le c$, and if $\mca^M_N(I) \subseteq V(I)$, then
    \[
    v \big(M/I^n N \big) =  v( 0 :_M I) \; \mbox{ for all } n \gg 0.
    \]
    \item 
    Let $(0 :_M I) = 0$, and  $ \emptyset \neq \mca^M_N(I) \subseteq V(I)$. Then,
    $$v \big(M/I^n N \big) = v \big(I^{n-1} N/I^n N \big) \; \mbox{ for all } n \gg 0,$$
    which is linear in $n$ with the leading coefficient in $ \{d_1, \dots, d_c \}$.
\end{enumerate}
\end{theorem}

\begin{proof}
    (1) is shown in Theorem~\ref{th:main1}. To prove (2), let $\fp \in \mca^M_N(I) \cap \Ass_R(0 :_M I)$. Then $I \subseteq \fp$. Moreover, there are two possibilities: either $\fp \notin \mcb_N(I)$ or $\fp \in \mcb_N(I)$. In the first case, the required result is shown in Theorem~\ref{th:main2i}, while the second case follows from Theorem~\ref{th:main2ii}. Finally, (3) and (4) are exactly Theorem~\ref{th:main3and4}.(1) and (2), respectively.
\end{proof}


\begin{remark}\label{rmk:depend-on-M-N}
    In Theorem~\ref{th:main}, we observe the following.
    \begin{enumerate}[\rm(1)]
    \item Part (2) (resp., (3)) shows that when $( 0 :_M I) \neq 0$, we eventually have $v_{\fp}(M/I^n N) < v_{\fp}(I^{n-1}N/I^n N)$ (resp., $v(M/I^n N) < v(I^{n-1} N/I^n N))$. Moreover, the functions $v_{\fp}(M/I^n N)$ and $v(M/I^n N)$ eventually become independent of $N$.
    \item
    Parts (1) and (4) improve the results of \cite[Thm.~2.14.(1) and (2)]{FG25a} by showing that, asymptotically, the linear functions $v_{\fp}(M/I^nN)$ and $v(M/I^nN)$ depend only on $I$ and $N$, and are independent of $M$.
    \end{enumerate}
\end{remark}

Next, we make the following observations regarding the behaviour of $v_{\fp}(M/I^nN)$ when $\fp \in \mca^M_N(I)\smallsetminus V(I)$.

\begin{proposition}\label{prop:vnumber}
    With \textnormal{Setup~\ref{setup}} and \textnormal{Notation~\ref{notation}}, let $\fp \in \mca^M_N(I)\smallsetminus V(I)$. Then $\fp \in \Ass_R(M/I^nN)$ for all $n\ge 0$. Moreover, the following hold:
    \begin{enumerate}[\rm (1)]
    \item 
    $v_{\fp}(M/I^n N) \le v_{\fp}(M/I^{n+1} N)$ for all $n \ge 0$.
    \item 
    $v_{\fp}(M/I^n N) \le \indeg \big((I + \fp)/ \fp \big) \cdot n + v_{\fp}(M/N)$ for all $n \ge 1$.
    \end{enumerate}
\end{proposition}

\begin{proof}
    In view of Proposition~\ref{prop:ass},
    \[
    \Ass_R(M/I^nN) \smallsetminus V(I) = \Ass_R(M/N) \smallsetminus V(I) \; \mbox{ for all } n \ge 0.
    \]
    Hence $\mca^M_N(I)\smallsetminus V(I) = \Ass_R(M/N) \smallsetminus V(I)$, and $\fp \in \Ass_R(M/I^nN)$ for all $n\ge 0$.
    
    (1) Let $n \ge 0$. Consider the short exact sequence
    \begin{equation*}
    0 \to \frac{I^n N}{I^{n+1}N} \to \frac{M}{I^{n+1}N} \to \frac{M}{I^{n}N} \to 0.
    \end{equation*}
    Since $I \nsubseteq \fp$, it follows that $\fp \notin \Ass_R \big(I^nN/I^{n+1}N \big)$. Consequently, the desired inequality follows from Proposition~\ref{prop:sesvnumber}.(2).
    
    (2) Since $\fp \in \Ass_R(M/N)$, there exists a homogeneous element $x \in M$ such that $(N :_R x) = \fp$ and $v_{\fp}(M/N) = \deg(x) $. Pick a homogeneous element $r \in I \smallsetminus \fp$ of the least possible degree. Then, for each $n \ge 1$, we have $r^nx\neq 0$ and $(I^n N :_R r^nx) = \fp$. So
    $$v_{\fp}(M/I^n N) \le \deg(r^nx) \le \deg(r)\cdot n + \deg(x) \; \mbox{ for all } n\ge 1.$$
    Hence the result follows, since $\deg(r) = \indeg((I + \fp)/ \fp)$ and $\deg(x)=v_{\fp}(M/N)$.
\end{proof}

\begin{remark}\label{rmk:gen-ineq}
\begin{enumerate}[\rm (1)]
    \item 
    Proposition~\ref{prop:vnumber} ensures that for $\fp \in \mca^M_N(I) \smallsetminus V(I)$, the function $v_{\fp}(M/I^n N) $ is always non-decreasing, and is bounded above by the linear function $\indeg((I + \fp)/ \fp) \cdot n + v_{\fp}(M/N)$. Thus, if $v_{\fp}(M/I^n N) $ is eventually linear, then its slope is at most $\indeg((I + \fp)/ \fp)$. This upper bound is sharp; see Example~\ref{ex:3} and Remark~\ref{rmk:slopebound}.(2).
    \item 
    The asymptotic behaviour of $v_{\fp}(M/I^nN)$ as a function of $n$ is not uniform across all $\fp \in \mca^M_N(I)\smallsetminus V(I)$, as illustrated in Example~\ref{ex:3}.
\end{enumerate}
\end{remark}

\section{Examples}\label{sec:examples}

In this section, we present several concrete examples to complement our results. The first example demonstrates that if the condition that $I$ is generated by homogeneous elements of positive degree is removed from Theorem~\ref{th:main}.(3), then the equality $v(M/I^nM) = v(0 :_M I)$ for $n \gg 0$ need not hold whenever $(0 :_M I) \neq 0$.

\begin{example}\label{ex:1}
    Let $R := \bz[X]$ be a standard graded polynomial ring in the variable $X$. Consider $M := R/(pX)$, $I = \fp := (p)$ and $\fm := (p,X)$, where $p$ is a prime number. Then $(0 :_M I) = (X)M \neq 0$. Moreover, the following hold.
    \begin{enumerate}[\rm (1)]
    \item 
    $\Ass_R(M/IM) = \Ass_R(0 :_M I) = \{ \fp \}$. For all $n \ge 2$, we have
    \[
    \Ass_R \big(M/I^nM\big) = \{ \fp, \fm \} \; \mbox{ and } \; \Ass_R \big(I^{n-1}M/I^nM \big) = \{ \fm\}.
    \]
    Thus, the sets $\Ass_R(0 :_M I)$ and $\mcb_M(I)$ are disjoint.
    \item 
    $v_{\fp}(M/IM) = 0$, and $v_{\fp} \big(M/I^nM \big) = v_{\fp}(M) = v_{\fp}(0 :_M I) = 1$ for all $n \ge 2$.
    \item 
    $v_{\fm}\big(M/I^nM \big) = v_{\fm} \big(I^{n-1}M/I^nM \big) = 0$ for all $n \ge 2$.
    \item 
    $v \big(M/I^nM \big) = v \big(I^{n-1}M/I^nM \big) = 0 < v(0 :_M I) = 1$ for all $n \ge 1$.
    \end{enumerate}
\end{example}

\begin{proof}
    Since $((pX) :_R p) = (X)$, we have $(0 :_M I) = (X)M \cong (R/\fp)(-1) \neq 0$.
    
    (1) Note that $M/IM \cong R/\fp$. Therefore $\Ass_R(M/IM) = \{ \fp \} = \Ass_R(0 :_M I)$. Let $n \ge 1$. Then $I^n M = (p^n, pX)/ (pX)$. Thus, 
    $$M/I^{n+1} M \cong R/(p^{n+1} ,pX) \; \mbox{ and } \; I^{n}M/I^{n+1} M \cong (p^{n} ,pX)/(p^{n+1} ,pX) .$$
    Now $(p^{n+1} , pX) = (p) \cap (p^{n+1} ,X) $, which is a primary decomposition of the ideal $(p^{n+1},pX)$ in $R$. Hence $\Ass_R \big(M/I^{n+1}M\big) = \{ \fp, \fm \}$. Since the maximal ideal $\fm$ annihilates the module $I^nM/I^{n+1} M$, it follows that $\Ass_R \big(I^nM/ I^{n+1}M \big) = \{ \fm \}$.
    
    (2) Since $M/IM \cong R/\fp$ and $(0 :_M I) \cong (R/\fp)(-1)$, one obtains $v_{\fp}(M/IM) = 0$ and $v_{\fp}(0 :_M I)=1$. Note that $\fp = ((pX) :_R X )$ and $\fp = ((p^n,pX) :_R X )$ for all $n \ge 2$, and these equalities do not hold if $X$ is replaced by any other homogeneous element of $R$ of lower degree. So
    $$v_{\fp} \big(M/I^nM \big) = 1 = v_{\fp}(M) = v_{\fp}(0 :_M I) \; \mbox{ for all } n \ge 2.$$
    
    (3) Let $n \ge 2$. Observe that $\fm = \big((p^n,pX) :_R p^{n-1} \big)$, and $I^{n-1} M/I^nM$ is generated by the image of $p^{n-1}$. So $v_{\fm}(I^{n-1} M/I^nM) = 0$. Hence,
    since $v_{\fm}(M/I^nM) \le v_{\fm}(I^{n-1} M/I^nM)$, it follows that $v_{\fm}(M/I^nM) = 0$.
    
    (4) It follows from (1), (2) and (3).
\end{proof}

In the following example, we show that the union $ \Ass_R(0:_M I) \cup \mcb_N(I)$ in Theorem~\ref{th:ass} need not be disjoint. 

\begin{example}\label{ex:2}
    Let $R := K[X,Y,Z]$ be a standard graded polynomial ring in three variables over a field $K$. Consider $I: = (XY,Z)$, $\fm : =(X,Y,Z)$, and $M := R/(X^3,Y,XZ)$. Then $(0 :_M I) = (X)M \neq 0$. Moreover, the following hold.  
    \begin{enumerate}[\rm (1)]
    \item 
    $\Ass_R \big(M/I^{n}M \big) = \Ass_R \big(I^{n-1}M/I^{n}M \big) = \Ass_R(0 :_M I) = \{ \fm \}$ for all $n \ge 1$. Thus, the sets $\Ass_R(0 :_M I)$ and $\mcb_M(I)$ are not disjoint.
    
    \item $v(0 :_M I) = v_{\fm}(0 :_M I) = 2$, and 
    \begin{align*} 
    v\big(M/I^{n} M \big) =  v_{\fm} \big(M/I^{n} M \big) = 
    \left \{  
    \begin{array}{ll} 
    2 & \text{if $n=1$ or $n \ge3$} \\ 
    1 & \text{if $n=2$}. 
    \end{array}
    \right.
    \end{align*}
    
    \item  
    $ v\big(I^{n-1}M/I^{n} M \big) =  v_{\fm} \big(I^{n-1}M/I^{n} M \big) = 
    \left \{  
    \begin{array}{ll} 
    2 & \text{if $n=1$} \\ 
    n-1 & \text{if $n\ge2$}. 
    \end{array}
    \right.
    $
    \end{enumerate}
\end{example}

\begin{proof}
    Since $ \big( (X^3,Y,XZ):_R (XY,Z) \big) = (X)$, we have 
    $$ (0 :_M I) = (X)M = (X,Y)/(X^3,Y,XZ) \neq (0).$$ 
    
    (1) Note that $\fm$ annihilates the module $(0 :_M I)$. So $\Ass_R(0 :_M I) =\{ \fm \}$. Now $I^{n} M = \big(I^{n} +(X^3,Y,XZ) \big) /(X^3,Y,XZ)$, and the monomial ideal $I^{n} + (X^3,Y,XZ)$ is generated by the elements $X^3,Y,XZ,Z^{n}$, where $n \ge 1$. Therefore, $M/IM \cong R/(X^3,Y,Z)$. For $n\ge 2$, we have
    $$\frac{M}{I^{n}M} \cong \frac{R}{(X^3,Y,XZ,Z^{n})} \; \text{ and } \; \frac{I^{n-1}M}{I^{n}M} \cong \frac{(X^3,Y,XZ,Z^{n-1})}{(X^3,Y,XZ,Z^{n})}.$$
    For $n\ge 1$, since $M/I^{n}M$ (and hence $I^{n-1}M/I^{n}M$) is annihilated by some power of $\fm$, it follows that $\Ass_R(M/I^{n}M) = \Ass_R(I^{n-1}M/I^{n}M) = \{\fm \}$.
    
    (2) Note that $\big((X^3, Y, XZ) :_R X^2 \big)= \fm$ and $\big((X^3, Y, Z) :_R X^2\big)= \fm$, and $X^2$ is a non-zero homogeneous element of least degree in $R$ with these properties. So
    \[
    v(0 :_M I) = v_{\fm}( 0:_M I) = 2 = v_{\fm}(M/IM) = v(M/IM).
    \]
    It can be verified that
    \[
    \big( (X^3, Y, XZ, Z^2) :_R Z \big) = \fm = \big((X^3,Y,XZ,Z^{n} ) :_R X^2 \big) \mbox{ for }n \ge 3,
    \]
    where the residue class of $Z$ in $M/I^2M$, and that of $X^2$ in $M/I^{n}M$ (for $n \ge 3$) are non-zero homogeneous elements of the least possible degree with these properties. Consequently, the statements in (2) follow.
    
    (3) The case $n=1$ follows from (2). So we fix $n \ge 2$. The module $I^{n-1} M/ I^{n}M$ is a non-zero cyclic $R$-module generated by the residue class of $Z^{n-1}$, and $\fm = \big((X^3,Y,XZ,Z^{n}) :_R Z^{n-1} \big)$. Therefore, the desired result follows.
\end{proof}

The example below shows that when $\fp \in \mca^M_N(I) \smallsetminus V(I)$, asymptotically, the function $v_{\fp}(M/I^nN)$ need not be a polynomial in $n$ of degree one.

\begin{example}\label{ex:3}
    Let $R = K[X,Y,Z]$ be a standard graded polynomial ring over a field $K$. Consider $I := (X,Y^2,Z^3)$, $\fp := (X)$, $\fq := (X,Y)$, $\fm := (X,Y,Z)$, $M := R/(X^3,XY^4)$, and $N := (X^3,XY)/(X^3,XY^4)$. Then $(0 :_M I) = 0$. Moreover, the following hold.
    \begin{enumerate}[\rm (1)]
    \item $\Ass_R(M/N) = \{ \fp, \fq \}$, and $\Ass_R(M/I^n N) = \{\fp, \fq ,\fm \}$ for all $n \ge 3.$ 
    \item 
    $v_{\fp}(M/I^n N) = 4$, $v_{\fq}(M/I^n N) = 3n-1$ and $v_{\fm}(M/I^n N) = 3n-2$ for all $n\ge 3$.
    \item 
    $v (M/I^n N) = 4$ for all $n \ge 3$.
    \end{enumerate}
\end{example}

\begin{proof}
    Since $Z^3 \in I$ is $M$-regular, we have $(0 :_M I) = 0$.
    
    (1) Note that $M/ N \cong R/(X^3, XY)$. Since $(X^3,XY) = (X) \cap (X^3,Y)$, it follows that $\Ass_R(M/N) = \{ \fp, \fq \}$. Denote 
    $$J_n := \big((X^3, XY^4) + I^n(X^3, XY) \big) \; \text{for all } n\ge 1.$$ 
    Then  $I^n N = J_n/(X^3, XY^4)$, and $M/ I^n N \cong R/ J_n$. Let $n \ge 3$. It can be seen that
    $$J_n = \big(X^3, XY^4, X^2Y^3Z^{3(n-2)}, X^2YZ^{3(n-1)},XY^3Z^{3(n-1)} , XYZ^{3n}\big).$$ 
    Then, splitting the last generator of $J_n$, we obtain
    \begin{align*}
    J_n 
    & = (X) \cap (X^3, Y) \cap (X^3, XY^4, X^2Y^3Z^{3(n-2)}, X^2YZ^{3(n-1)},XY^3Z^{3(n-1)} , Z^{3n}) \\
    & = (X) \cap (X^3,Y) \cap \big(X^3,Y^4, X^2Y^3Z^{3(n-2)}, X^2Z^{3(n-1)}, Y^3Z^{3(n-1)}, Z^{3n}\big).
    \end{align*}
    The second equality is obtained after splitting $XY^4$, $X^2YZ^{3(n-1)}$ and $XY^3Z^{3(n-1)}$ one by one.
    Thus we get a primary decomposition of $J_n$. Consequently, $$\Ass_R(M/I^nN) = \{ \fp, \fq, \fm \}.$$
    
    (2) Fix $n \ge 3$. Let us denote $x,y,z$ the residue classes of $X,Y,Z$ modulo $(X^3,XY^4)$, respectively. For an ideal $\fa$ of $R$, note that
    \[
     (I^n N :_M \fa) = (J_n :_R \fa)/(X^3,XY^4).
    \]
    This helps us to compute the following:
    \begin{equation}\label{eq:colonvnumber}
    \begin{aligned}
    (I^n N :_M \fp)  = & \big(x^2, y^4, xy^3z^{3(n-2)},xyz^{3(n-1)},
    y^3z^{3(n-1)},yz^{3n} \big), \\
    (I^n N :_M \fq) = & \big(x^3,xy^4,x^2y^3, xy^3z^{3(n-2)}, x^2y^2z^{3(n-2)}, x^2z^{3(n-1)}, \\
    & xy^2z^{3(n-1)}, xyz^{3n} \big) \\
    (I^n N :_M \fm) = & \big(x^3,xy^4, x^2y^3z^{3n-7}, x^2yz^{3(n-1)}, xy^3z^{(3n-4)}, x^2y^2z^{3n-4}, \\
    & xyz^{3n}, xy^2z^{3n-1} \big).
    \end{aligned}
    \end{equation}
    Observe that if $\fa$ is an ideal of $R$ and $u \in M$ satisfying $\fa = (I^n N :_R u)$, then necessarily $u \in (I^n N:_M \fa)$. Consequently, by \eqref{eq:colonvnumber}, one obtains
    $\fp = \big(I^n N :_R y^4 \big)$, $\fq = \big(I^n N :_R x^2z^{3(n-1)} \big)$, and $\fm = \big(I^n N :_R x^2y^3z^{3n-7} \big)$. Moreover, these equalities will not hold if one replaces the elements $y^4, x^2z^{3(n-1)}$ and $x^2y^3z^{3n-7}$ by any other smaller degree homogeneous elements from the submodules in equation \eqref{eq:colonvnumber}, respectively. This proves our assertion in (2). 
    
    (3) It is clear from (2).
\end{proof}

\begin{remark}\label{rmk:slopebound}
\begin{enumerate}[\rm(1)]
\item 
Example~\ref{ex:3} shows that when $N \neq M$, the function $v(M/I^nN)$ might be asymptotically constant even though $(0 :_M I) = 0$. This situation is unlikely to occur when $N=M$, see Theorem~\ref{th:intro2}.(4). 
\item 
In Example~\ref{ex:3}, $\mca^M_N(I) \smallsetminus V(I) = \{ \fp, \fq \}.$ The function $v_{\fp}(M/I^nN)$ is asymptotically constant, whereas the function $v_{\fq}(M/I^nN)$ is asymptotically linear with slope $3$. Here $(I + \fq)/\fq = (X,Y,Z^3)/(X,Y)$. So $\indeg((I + \fq)/\fq) = 3$.
\end{enumerate}
\end{remark}

Based on experimental evidence, we expect that for $ \fp \in \mca^M_N(I) \smallsetminus V(I)$, the function $v_{\fp}(M/I^n N)$ is asymptotically constant or a polynomial in $n$ of degree one. Thus, we end this article by posing the following natural question.

\begin{question}\label{remaining-ques}
For each $\fp \in \mca^M_N(I) \smallsetminus V(I)$, is it true that the function $v_{\fp}(M/I^n N)$ is eventually either constant or a polynomial in $n$ of degree one?
\end{question}

\bibliographystyle{plain}
\bibliography{main}

\end{document}